\documentclass[a4paper, 12pt]{amsart}

\usepackage{amssymb}

\mathchardef\emptyset="001F

\theoremstyle{plain}

\newtheorem{theorem}{Theorem}[section]
\newtheorem{lemma}[theorem]{Lemma}
\newtheorem{proposition}[theorem]{Proposition}

\newtheorem{remark}[theorem]{Remark}

\theoremstyle{definition}
\theoremstyle{remark}
\numberwithin{equation}{section}
\newcommand{\e}{\varepsilon}

\newcommand{\R}{{\mathbb R}}

\newcommand{\tr}{{\rm tr}\,}

\newcommand{\C}{{\mathbb C}}
\newcommand{\Ao}{{\mathbb A}}

\newcommand{\T}{{\mathbb T}}
\newcommand{\psin}{\psi_{n}}
\newcommand{\bpsin}{\bar{\psi}_{n}}
\newcommand{\bphi}{\bar{\phi}}
\newcommand{\var}{\varphi}
\newcommand{\bvar}{\bar{\varphi}}
\newcommand{\no}{\noindent}
\newcommand{\non}{\nonumber}
\newcommand{\lan}{\lambda_{n}}
\renewcommand{\div}{\hbox{{\rm div}}}
\newcommand{\A}{A^{\e}}
\newcommand{\psie}{\psin^{\e}}
\newcommand{\bpsie}{\bpsin^{\e}}
\newcommand{\ce}{c^{\e}}
\newcommand{\lae}{\lan^{\e}}
\newcommand{\dsp}{\displaystyle}

\title[]
{Localization for the Schr\"{o}dinger equation in 
a locally periodic medium}
\author[Gr\'egoire Allaire]{Gr\'egoire Allaire$^{1}$}
\author[Mariapia Palombaro]{Mariapia Palombaro$^{2}$}

\begin{document}
\baselineskip3.15ex
\vskip .3truecm

\maketitle
{
\small
\noindent
$^1$ Centre de Math\'ematiques Appliqu\'ees, \'Ecole Polytechnique, 
91128 Palaiseau, France.\\
Email: gregoire.allaire@polytechnique.fr\\
\noindent
$^2$ Centre de Math\'ematiques Appliqu\'ees, \'Ecole Polytechnique, 
91128 Palaiseau, France.\\
Email: mariapia.palombaro@polytechnique.fr 
}
\begin{abstract}
\small{
We study the homogenization of a Schr\"{o}dinger equation in a 
locally periodic medium. For the time and space scaling of 
semi-classical analysis we consider well-prepared initial data 
that are concentrated near a stationary point (with respect 
to both space and phase) of the energy, i.e. the Bloch cell 
eigenvalue. We show that there exists a localized solution 
which is asymptotically given as the product of a Bloch
wave and of the solution of an homogenized Schr\"{o}dinger
equation with quadratic potential.

\vskip.3truecm
\noindent  {\bf Key words:}
 Homogenization, localization, Bloch waves, Schr\"{o}dinger.
\vskip.2truecm
\noindent  {\bf 2000 Mathematics Subject Classification:} 
35B27, 35J10.
}
\end{abstract}

\section{introduction}

\no
We study the homogenization of the following Schr\"{o}dinger equation

\begin{equation}\label{start}
\left\{
\begin{array}{ll}
\displaystyle 
\frac{i}{\e}\frac{\partial u_{\e}}{\partial t} - 
\div\left(A\left(x,\frac{x}{\e}\right)\nabla u_{\e}\right) 
+ \frac{1}{\e^{2}} c\left(x,\frac{x}{\e}\right) u_{\e}=0 & \mbox{ in } \R^{N}\times \R^+\\
[3mm]
u_{\e}(0,x)=u^{0}_{\e}(x) \hspace{5.1cm} & \mbox{ in } \R^{N}\\
\end{array}
\right.
\end{equation}

\no
where the unknown $u_{\e}(t,x)$ is a complex-valued function.  
The coefficients $A(x,y)$ and $c(x,y)$ are real and sufficiently smooth   
bounded functions defined for $x\in \R^{N}$ (the macroscopic variable) 
and $y\in \T^{N}$ (the microscopic variable in the unit torus).
The period $\e$ is a small positive parameter which is 
intended to go to zero. 
Furthermore the matrix $A$ is symmetric, uniformly positive definite. 
Of course the usual Schr\"{o}dinger equation is recovered when 
$A\equiv Id$ but, since there is no additional difficulty, we keep
the general form of equation (\ref{start}) in the sequel (which can 
be interpreted as introducing a non flat locally periodic metric). 
\par
The scaling of (\ref{start}) is that of semi-classical analysis
(see e.g. \cite{blp}, \cite{buslaev}, \cite{guillot},
\cite{gerard}, \cite{gerard2}, \cite{sjostrand}, \cite{guillot2},
\cite{pst}, \cite{pr}): if the period is rescaled to 1, it amounts
to look at large, time and space, variables of order $\e^{-1}$.
At least in the case when $A\equiv Id$ and $c(x,y)=c_0(x)+c_1(y)$,
there is a well-known theory for the asymptotic
limit of (\ref{start}) when $\e$ goes to zero. By using
WKB asymptotic expansion or the notion of semi-classical measures 
(or Wigner transforms) the homogenized problem is in some sense 
the Liouville transport equation for a classical particle which 
is the limit of the wave function $u_\e$. In other words,
for an initial data living in the $n$-th Bloch band and
under some technical assumptions on the Bloch spectral cell
problem (\ref{celleq}),
the semi-classical limit of (\ref{start}) is given by the
dynamic of the following Hamiltonian system in the phase space
$(x,\theta)\in\R^N\times\T^N$
\begin{equation}
\label{hamilton}
\left\{ \begin{array}{l}
\dot x = \nabla_\theta \lambda_n(x,\theta) \\
\dot \theta = - \nabla_x \lambda_n(x,\theta)
\end{array} \right.
\end{equation}
where the Hamiltonian $\lambda_n(x,\theta)$ is precisely
the $n$-th Bloch eigenvalue of (\ref{celleq}) (see \cite{buslaev},
\cite{guillot}, \cite{gerard}, \cite{gerard2}, \cite{sjostrand},
\cite{guillot2}, \cite{pst}, \cite{pr} for more details).
\par
Our approach to (\ref{start}) is different since we consider
special initial data that are monochromatic, have zero group
velocity and zero applied force. Namely the initial data is
concentrating at a point $(x^n,\theta^n)$ of the phase space
where $\nabla_\theta \lambda_n(x^n,\theta^n)=\nabla_x \lambda_n(x^n,\theta^n)=0$.
In such a case, the previous Hamiltonian system (\ref{hamilton})
degenerates (its solution is constant) and is unable to describe the precise
dynamic of the wave function $u_\e$. We exhibit another limit
problem which is again a Schr\"{o}dinger equation with quadratic 
potential. In other words we build a sequence of approximate
solutions of (\ref{start}) which are the product of a Bloch
wave and of the solution of an homogenized Schr\"{o}dinger
equation. Furthermore, if the full Hessian tensor of the
Bloch eigenvalue $\lambda_n(x,\theta)$ is positive definite
at $(x^n,\theta^n)$, we prove that all the eigenfunctions
of an homogenized Schr\"{o}dinger equation are exponentially
decreasing at infinity. In other words, we exhibit a localization
phenomenon for (\ref{start}) since we build a sequence of approximate
solutions that decay exponentially fast away from $x^n$. The
root of this localization phenomenon is the macroscopic modulation
(i.e. with respect to $x$) of the periodic coefficients which
is similar in spirit to the randomness that causes Anderson's 
localization (see \cite{cl} and references therein).
\par
Let us describe more precisely the type of well-prepared 
initial data that we consider. For a given point $(x^n,\theta^n)\in\R^N\times\T^N$
and a given function $v^{0}\in H^{1}(\R^{N})$ we take 
\begin{equation}\label{wp}
u_{\e}^{0}(x)=\psin\Big(x^{n},\frac{x}{\e},\theta^{n}\Big)
e^{2i\pi\frac{\theta^{n}\cdot x}{\e}}
v^{0}\Big(\frac{x-x^n}{\sqrt{\e}}\Big)
\end{equation}
where $\psin(x,y,\theta)$ is a so-called Bloch eigenfunction, 
solution of the following Bloch spectral cell equation

\begin{equation}\label{celleq}
- (\div_{y} + 2i\pi\theta)(A(x,y)(\nabla_{y} + 2i\pi\theta)\psin) + c(x,y) = 
\lan(x,\theta)\psin 
\hspace{1cm} \mbox{ in }\T^{N}\,,
\end{equation}

\no
corresponding to the $n$-th eigenvalue or energy level $\lan$. 
The Bloch wave $\psin$ is periodic with respect to $y$ but 
$v^0$ is not periodic, so $v^{0}\Big(\frac{x-x^n}{\sqrt{\e}}\Big)$
means that the initial data is concentrated around $x^n$ with 
a support of asymptotic size $\sqrt\e$. 
The Bloch frequency $\theta^n\in\T^{N}$, the localization point
$x^n\in\R^{N}$ and the energy level $n$ are chosen such that 
$\lan(x^n,\theta^n)$ is simple and $\nabla_x\lan(x^n,\theta^n) 
= \nabla_\theta\lan(x^n,\theta^n) =0$.  

Our main result (Theorem \ref{mainth}) shows that the solution 
of \eqref{start} is approximately given by
\begin{equation}\label{asymptot}
u_{\e}(t,x)\approx \psin\Big(x^{n},\frac{x}{\e},\theta^{n}\Big)
e^{i\frac{\lan(x^n,\theta^n) t}{\e}}e^{2i\pi\frac{\theta^{n}\cdot x}{\e}}
v\Big(t,\frac{x-x^n}{\sqrt{\e}}\Big)\,,
\end{equation}
where $v$ is the unique solution of the homogenized Schr\"odinger equation

\begin{equation}\label{hom0}
\left\{
\begin{array}{ll}
\displaystyle 
i \frac{\partial v}{\partial t} - \div\left(A^{*}\nabla v \right) 
+ \div(v B^{*} z) + c^{*} v 
+ v D^{*} z \cdot z  = 0 & \mbox{ in } \R^{N}\times \R^+\\
[3mm]
v (0,z)=v^{0}(z) \hspace{5.1cm} & \mbox{ in } \R^{N}\\
\end{array}
\right.
\end{equation}

\no
where $c^*$ is a constant coefficient and $A^*,B^*,D^*$ are constant 
matrices defined by
$$
A^{*}=\frac{1}{8\pi^{2}}\nabla_{\theta}\nabla_{\theta}\lan(x^n,\theta^n) \,, \
B^{*}=\frac{1}{2i\pi}\nabla_{\theta}\nabla_{x}\lan(x^n,\theta^n) \,, \
D^{*}=\frac{1}{2}\nabla_{x}\nabla_{x}\lan(x^n,\theta^n) \,.
$$
In Proposition \ref{propoself} we show that the homogenized
problem \eqref{hom0} is well-posed since the underlying
operator is self-adjoint. Furthermore, under the additional
assumption that the Hessian tensor $\nabla\nabla\lan(x^n,\theta^n)$
(with respect to both variables $x$ and $\theta$) is positive
definite, we prove that \eqref{hom0} admits a countable
number of eigenvalues and eigenfunctions which all decay
exponentially at infinity (see Proposition \ref{localization}).
In such a case, formula \eqref{asymptot} defines a family
of approximate (exponentially) localized solutions of
\eqref{start}.
\par
Let us indicate that the case of the first eigenvalue (ground state) 
$n=1$ with $\theta^1=0$ was already studied in \cite{alpiat1} (for 
the spectral problem rather than the evolution equation). 
The case of purely periodic coefficients (i.e. that depend only on 
$y$ and not on $x$) is completely different and was studied in 
\cite{alpiat2}. Indeed, in this latter case there is no localization 
effect and one proves that, for a longer time scale (of order $\e^{-1}$ 
with respect to \eqref{start}), the homogenized
limit is again a Schr\"{o}dinger equation without the drift and quadratic 
potential in \eqref{hom0}.

\section{Preliminaries}
\no
In the present section we give our main assumptions, set some notation 
and a few preliminary results needed in the proof of the main results in 
Section \ref{mainsec}. 

We first assume that the coefficients $A_{ij}(x,y)$ and
$c(x,y)$ are real, bounded, and Carath\'eodory functions 
(measurable with respect to $y$ and continuous in $x$), 
which are periodic with respect to $y$. In other words, 
they belong to $C_b\left(\R^N;L^\infty(\T^N)\right)$. 
Furthermore, the tensor $A(x,y)$ is symmetric uniformly coercive.
Under these assumptions, it is well-known that, for any values 
of the parameters ${\theta\in\T^{N}}$ and ${x\in\R^{N}}$, 
the cell problem \eqref{celleq} defines 
a compact self-adjoint operator on $L^{2}(\T^{N})$ which admits a 
countable sequence of real increasing eigenvalues  
${\displaystyle \{\lan(x,\theta)\}_{n\geq 1}}$ (repeated with their multiplicity) 
with corresponding eigenfunctions 
${\displaystyle \{\psin(x,\theta,y)\}_{n\geq 1}}$ 
normalized by 
$$
||\psin(x,\theta,\cdot)||_{L^{2}(\T^{N})} =1\,.
$$

Our main assumptions are:

\no
{\bf Hypothesis H1.} 
There exist $x^{n}\in\R^{N}$ and $\theta^{n}\in\T^{N}$ such that

\begin{equation}\label{assumpt}
\begin{cases}
&\hspace{-3mm}(i)\:  \lan(x^{n},\theta^{n})  \text{ is a simple eigenvalue,} \\
&\hspace{-3mm}(ii)\:  (x^{n},\theta^{n})  \text{ is a critical point of } \lan(x,\theta),
i.e. \: \nabla_{x}\lan(x^{n},\theta^{n})=\nabla_{\theta}\lan(x^{n},\theta^{n})=0.
\\
\end{cases}
\end{equation}

\vspace{2mm}

\no 
{\bf Hypothesis H2.}  
The coefficients $A(x,y)$ and $c(x,y)$ are of class $C^{2}$ 
with respect to the variable $x$ in a neighborhood of $x=x^{n}$.

\vspace{2mm}

\no 
Then we set:
\begin{equation*}
A_{1,h}(y):=\frac{\partial A}{\partial x_{h}}(x^{n},y)\,,\quad
A_{2,lh}(y):=\frac{\partial^{2} A}{\partial x_{l}\partial x_{h}}(x^{n},y)\,,\quad
\text{ for } \: l,h=1,\dots,N\,.
\end{equation*}
Similar notation is used to denote the derivatives of the function $c$ with 
respect to the $x$-variable. 
With an abuse of notation we further set 
$$
A(y):=A(x^{n},y)\,,\quad \lan:=\lan(x^{n},\theta^{n})\,,\quad 
\psin(y):=\psin(x^{n},y,\theta^{n})\,,
$$
and analogous notation holds for all derivatives of $\psin$ and $\lan$ with 
respect to the $x$-variable and the $\theta$-variable evaluated at 
$x=x^{n}$ and $\theta=\theta^{n}$.
Without loss of generality we will assume in the sequel that $x^{n}=0$.\\
{\bf Notation.} For any function $\rho(y)$ defined on $\T^{N}$ we set  
$$
\rho^{\e}(z):=\rho(z/\sqrt{\e})
$$ 
where $z:=\sqrt{\e}y\equiv x/\sqrt\e$.  
In the sequel the symbols $\div_y$ and $\nabla_y$ will stand for the divergence 
and gradient operators which act  
with respect to the $y$-variable while div and $\nabla$ will indicate the divergence and 
gradient operators which act with respect to the $z$-variable.
Finally throughout this paper the Einstein summation convention is used.


\medskip

Under assumption \eqref{assumpt}-(i) it is a classical matter to prove that 
the $n$-th eigencouple of \eqref{celleq} is smooth with respect to the variable 
$\theta$ in a 
neighborhood of $\theta=\theta^n$ (see \cite{kato}) and has the same differentiability 
property as the coefficients with respect to the variable $x$. 
Introducing the unbounded operator $\Ao_n(x,\theta)$ defined on $L^2(\T^N)$ by
\begin{equation*}
\Ao_n(x,\theta)\psi = 
- (\div_y +2i\pi\theta) \Big( A(x,y) (\nabla_y +2i\pi\theta)\psi \Big) + 
c(x,y) \psi - \lambda_n(x,\theta) \psi ,
\end{equation*}
it is easy to differentiate \eqref{celleq}.
Denoting by $(e_k)_{1\leq k\leq N}$ the canonical basis
of $\R^N$, the first derivatives satisfy
\begin{equation}
\label{deriv1t}
\begin{array}{ll}
\dsp \Ao_n(x,\theta)\frac{\partial\psin}{\partial\theta_k} = 
&\dsp 2i\pi e_k  A(x,y)(\nabla_y +2i\pi\theta)\psin  \\[0.3cm] 
&\dsp + (\div_y +2i\pi\theta) \left( A(x,y) 2i\pi e_k \psin \right) +
\frac{\partial\lambda_n}{\partial\theta_k}(x,\theta) \psin \,,
\end{array}
\end{equation}

\begin{equation}
\label{deriv1x}
\begin{array}{ll}
\dsp \Ao_n(x,\theta)\frac{\partial\psin}{\partial x_l} =
&\dsp  (\div_y +2i\pi\theta) \Big( \frac{\partial A}{\partial x_l}(x,\theta)
       (\nabla_y +2i\pi\theta)\psin \Big) \\[0.3cm] 
&\dsp  -\frac{\partial c}{\partial x_l}(x,y)\psin
       + \frac{\partial \lambda_n}{\partial x_l}(x,\theta)\psin \,.
\end{array}
\end{equation}

\no
Similar formulas hold for second order derivatives. By integrating 
the cell equations for the second order derivatives against $\psin$ 
we obtain the following formulas that will be useful in the sequel
(their proofs are safely left to the reader).

\begin{lemma}
Assume that assumptions {\bf H1} and {\bf H2} hold true.
Then the following equalities hold: 

\begin{align}\label{eq4}
        &  \int_{\T^{N}} \frac{1}{2\pi i}\Big[ 
            A_{1,h}(\nabla_{y}+2i\pi\theta^{n})\frac{\partial \psin}{\partial \theta_{k}} 
            \cdot (\nabla_{y} -2i\pi\theta^{n})\bpsin + 
            c_{1,h}\,\frac{\partial \psin}{\partial \theta_{k}} \bpsin\Big] \: dy \\
\non &   +\int_{\T^{N}}\Big[
             A_{1,h} e_{k}\psin \cdot (\nabla_{y}-2i\pi\theta^{n})\bpsin
            + A e_{k}\frac{\partial \psin}{\partial x_{h}}\cdot 
         (\nabla_{y}-2i\pi\theta^{n})\bpsin \Big]\: dy \\
\non &   -\int_{\T^{N}}\Big[
            e_{k}\bpsin A_{1,h}\cdot (\nabla_{y}+2i\pi\theta^{n})\psin 
           + e_{k}\bpsin A \cdot 
          (\nabla_{y}+2i\pi\theta^{n})\frac{\partial \psin}{\partial x_{h}}\Big]\: dy \\ 
\non &   - \frac{1}{2i\pi}\frac{\partial^{2} \lan}{\partial x_{h}\partial \theta_{k}} = 0\,,
\end{align}

\begin{align}\label{eq5}
       &  \int_{\T^{N}} \Big[ 
          A_{2,lh}(\nabla_{y}+2i\pi\theta^{n})\psin\cdot (\nabla_{y} -2i\pi\theta^{n})\bpsin
           + \Big(c_{2,lh} - \frac{\partial^{2} \lan}{\partial x_{l}\partial x_{h}}\Big)
          |\psin|^{2} \Big] \: dy  \\
\non &  + \int_{\T^{N}} \Big[ 
           A_{1,h}(\nabla_{y}+2i\pi\theta^{n})\frac{\partial \psin}{\partial x_{l}}
           \cdot (\nabla_{y} -2i\pi\theta^{n})\bpsin +
           c_{1,h}\,\frac{\partial \psin}{\partial x_{l}}\bpsin \Big]\: dy \\ 
\non &  + \int_{\T^{N}} \Big[
           A_{1,l}(\nabla_{y}+2i\pi\theta^{n})\frac{\partial \psin}{\partial x_{h}}\cdot 
           (\nabla_{y} -2i\pi\theta^{n})\bpsin  
           + c_{1,l}\,\frac{\partial \psin}{\partial x_{h}}\bpsin  \Big] 
           \: dy = 0\,,
\end{align}

\begin{align}\label{eq6}
        &  \int_{\T^{N}} \Big[
           2i\pi e_k  A(y) 
          (\nabla_y +2i\pi\theta^{n})\frac{\partial\psi_n}{\partial\theta_l}\bpsin 
          - \left( A(y) 2i\pi e_k 
          \frac{\partial\psi_n}{\partial\theta_l} \right)(\nabla_y -2i\pi\theta^{n})\bpsin 
          \Big]\: dy \\
\non & + \int_{\T^{N}} \Big[
         2i\pi e_l  A(y) 
         (\nabla_y +2i\pi\theta^{n})\frac{\partial\psi_n}{\partial\theta_k}\bpsin
        -\left( A(y) 2i\pi e_l \frac{\partial\psi_n}{\partial\theta_k} \right)
        (\nabla_y -2i\pi\theta^{n})\bpsin \Big]\: dy \\
\non & -\int_{\T^{N}} \Big[
         4\pi^2e_k A(y) e_l |\psi_n|^{2} +4\pi^2e_l A(y) e_k |\psi_n|^{2} \Big]\: dy \\
\non &  + \frac{\partial^2\lambda_n}{\partial\theta_l\partial\theta_k}(\theta^{n}) =0\,.
\end{align}
\end{lemma}

We now give the variational formulations of the above cell problems,
rescaled at size $\e$.

\begin{lemma}
Assume that assumptions {\bf H1} and {\bf H2} hold true and let 
$\var(z)$ be a smooth compactly supported function defined from $\R^{N}$ 
into $\C$. 
Then the following equalities hold: 

\begin{align}\label{eq1}
       &  \int_{\R^{N}} \Big[ \A(\nabla_{y}+2i\pi\theta^{n})\psie\cdot
          (\sqrt{\e}\nabla-2i\pi\theta^{n})\bvar(z) +(\ce - \lae)\psie \bvar\Big] \: dz = 0 \,,
\end{align}

\begin{align}\label{eq2}
       &  \int_{\R^{N}}\Big[ \A(\nabla_{y}+2i\pi\theta^{n})
           \frac{\partial \psie}{\partial \theta^{n}_{k}}
           \cdot (\sqrt{\e}\nabla -2i\pi\theta^{n})\bvar + (\ce- \lae)
           \frac{\partial \psie}{\partial \theta^{n}_{k}} \bvar\Big] \: dz \\
\non &   +\int_{\R^{N}}\Big[ -2\pi i e_{k} \cdot \A(\nabla_{y}+2i\pi\theta^{n})
            \psie\bvar + \A \, 2\pi i e_{k} \psie \cdot 
          (\sqrt{\e}\nabla -2i\pi\theta^{n})\bvar\Big]  \: dz=0\,, 
\end{align}

\begin{align}\label{eq3}
       &  \int_{\R^{N}}\Big[ \A(\nabla_{y}+2i\pi\theta^{n})
           \frac{\partial \psie}{\partial x_{h}}\cdot(\sqrt{\e}\nabla -2i\pi\theta^{n}) \bvar
            +(\ce - \lae) \frac{\partial \psie}{\partial x_{h}}\bvar \Big] \: dz  \\
\non &   +\int_{\R^{N}}\Big[
            \A_{1,h}(\nabla_{y}+2i\pi\theta^{n})\psie \cdot
         (\sqrt{\e}\nabla -2i\pi\theta^{n})\bvar +
            \ce_{1,h}\,\psie \bvar\Big] \: dz =0\,.
\end{align}

\end{lemma}

\begin{proof}
Formula \eqref{eq1} follows straightforwardly from equation \eqref{celleq} while 
\eqref{eq2}-\eqref{eq3} are consequences of \eqref{deriv1t}-\eqref{deriv1x}.
\end{proof}

Finally we recall the notion of two-scale convergence 
introduced in \cite{allaire}, \cite{nguetseng} (that will be used 
with $\delta=\sqrt\e$).

\begin{proposition}
\label{prop2s}
Let $f_{\delta}$ be a sequence uniformly bounded in $L^2(\R^N)$.
\begin{enumerate}
\item There exists a subsequence, still denoted by $f_\delta$, and a
limit $f_0(x,y) \in L^2(\R^N\times\T^N)$ such that
$f_\delta$ {\em two-scale converges} weakly to $f_0$ in the
sense that
$$
\lim_{\delta\to 0}
\int_{\R^N} f_\delta(x)\phi(x,x/\delta)\,dx =
\int_{\R^N}\int_{\T^N} f_0(x,y)\phi(x,y)\,dx\,dy
$$
for all functions $\phi(x,y)\in L^2\left( \R^N ; C(\T^N) \right)$. 

\item Assume further that $f_\delta$ two-scale converges 
weakly to $f_0$ and that
$$
\lim_{\delta\to 0} \| f_{\delta} \|_{L^2(\R^N)} = 
\| f_0 \|_{L^2\left(\R^N\times\T^N\right)} .
$$
Then $f_\delta$ is said to two-scale converge {\em strongly} to its
limit $f_0$ in the sense that, if $f_0$ is smooth enough, e.g. 
$f_0\in L^2\left( \R^N ; C(\T^N) \right)$, we have
$$
\lim_{\delta\to 0} \int_{\R^N} \left| f_\delta(x)-
f_0\right(x,x/\delta\left) \right|^2 dx = 0.
$$

\item Assume that $\delta \nabla f_\delta$ is also 
uniformly bounded in $L^2(\R^N)^N$. Then
there exists a subsequence, still denoted by $f_\delta$, and a limit
$f_0(x,y) \in L^2(\R^N ;H^1(\T^N))$ such that
$f_\delta$ two-scale converges to $f_0(x,y)$ and $\delta\nabla
f_\delta$ two-scale converges to $\nabla_y f_0(x,y)$.
\end{enumerate}
\end{proposition}


\section{Main results}\label{mainsec}
\no
We begin by recalling the usual a priori estimates for the solution 
of the Schr\"{o}dinger equation \eqref{start} which hold true since 
the coefficients are real. They are obtained by multiplying the 
equation successively by $\overline u_\e$ and 
$\frac{\partial \overline u_\e}{\partial t}$, and integrating by parts. 

\begin{lemma}\label{apriori}
There exists $C>0$ independent of $\e$ such that the solution of \eqref{start} satisfies
\begin{align*}
& ||u_{\e}||_{L^{\infty}(\R^+;L^{2}(\R^{N}))} = ||u_{\e}^{0}||_{L^{2}(\R^{N})}\,,\\
&  \e||\nabla u_{\e}||_{L^{\infty}(\R^+;L^{2}(\R^{N}))} \leq
C\Big( ||u_{\e}^{0}||_{L^{2}(\R^{N})}+\e ||\nabla u_{\e}^{0}||_{L^{2}(\R^{N})}\Big)\,.
\end{align*}
\end{lemma}

\begin{theorem}\label{mainth}
Assume that assumptions {\bf H1} and {\bf H2} hold true and that the initial data $u_{\e}^{0}$ 
is of the form \eqref{wp}. Then the solution of \eqref{start} can be written as 
\begin{equation}\label{form}
u_{\e}(t,x)=e^{i\frac{\lan t}{\e}}e^{2i\pi\frac{\theta^{n}\cdot x}{\e}}
v_{\e}\Big(t,\frac{x-x^n}{\sqrt{\e}}\Big)\,,
\end{equation}
where $v_{\e}(t,z)$ two-scale converges strongly to $\psin(y)v(t,z)$, i.e. 
\begin{equation}
\label{eq2s.6c}
\lim_{\e\to0} \int_{\R^N} \left| v_\e(t,z) -
\psin\left(\frac{z}{\sqrt\e}\right) v(t,z) \right|^2 dz = 0 ,
\end{equation}
uniformly on compact time intervals in $\R^+$, and 
$v$ is the unique solution of the homogenized Schr\"odinger equation

\begin{equation}\label{hom}
\left\{
\begin{array}{ll}
\displaystyle 
i \frac{\partial v}{\partial t} - \div\left(A^{*}\nabla v \right) 
+ \div(v B^{*} z) + c^{*} v 
+ v D^{*} z \cdot z  = 0 & \mbox{ in } \R^{N}\times \R^+\\
[3mm]
v (0,z)=v^{0}(z) \hspace{5.1cm} & \mbox{ in } \R^{N}\\
\end{array}
\right.
\end{equation}

\no
where 
$$
A^{*}=\frac{1}{8\pi^{2}}\nabla_{\theta}\nabla_{\theta}\lan(x^n,\theta^n) \,, \
B^{*}=\frac{1}{2i\pi}\nabla_{\theta}\nabla_{x}\lan(x^n,\theta^n) \,, \
D^{*}=\frac{1}{2}\nabla_{x}\nabla_{x}\lan(x^n,\theta^n) \,,
$$ 
and $c^{*}$ is given by
\begin{equation*}
c^{*}=\int_{\T^{N}}\hspace{-1mm}\Big[
          A(\nabla_{y}+2i\pi\theta^{n})\psin\cdot \frac{\partial \bpsin}{\partial x_{k}} e_{k}
         - A (\nabla_{y}-2i\pi\theta^{n})\frac{\partial \bpsin}{\partial x_{k}}\cdot \psin\,e_{k}
         - A_{1,k}(\nabla_{y}-2i\pi\theta^{n})\bpsin\cdot \psin e_{k} \Big]  dy \,. 
\end{equation*}

\end{theorem}

\begin{remark}
{\rm
Notice that even if the tensor $A^*$ might be  
non-coercive, the homogenized problem \eqref{hom} is well posed. 
Indeed the operator 
$\Ao^{*}:L^2(\R^N)\to L^2(\R^N)$ defined by 
\begin{equation}\label{herm}
\Ao^* \var = - \div\left(A^{*}\nabla \var \right) 
+ \div(\var B^{*} z) + c^{*} \var
+ \var D^{*} z \cdot z 
\end{equation}
is self-adjoint (see Proposition \ref{propoself}) and therefore 
by using semi-group theory (see {\em e.g.} \cite{brezis} or 
Chapter X in \cite{reedsimon}), one can show that there exists a unique solution 
in $C(\R^+;L^2(\R^N))$, although it may not belong to
$L^2(\R^+;H^1(\R^N))$.
}
\end{remark}

The next result establishes the conservation of the $L^2$-norm for the 
solution $v$ of the homogenized equation \eqref{hom} and the self-adjointness
of the operator $\Ao^*$. 

\begin{proposition}\label{propoself}
Let $v\in C(\R^+;L^2(\R^N))$ be solution to \eqref{hom}. Then
\begin{equation}\label{cons}
||v(t,\cdot)||_{L^2(\R^N)}=||v^0||_{L^2(\R^N)} \quad \forall \,t\in\R^+\,.
\end{equation} 
Moreover the operator $\Ao^{*}$ defined in \eqref{herm} is self-adjoint. 
\end{proposition}

\begin{proof}
We multiply the equation \eqref{hom} by $\bar{v}$ and take the imaginary part to 
obtain
\begin{equation}\label{impart}
\frac{1}{2}\frac{d}{dt}\int_{\R^N}|v|^2 \, dz =
{\rm Im}\left(\int_{\R^N} v B^*z\cdot \nabla\bar{v}-c^* |v|^2 \, dz \right)\,. 
\end{equation}
After integrating by parts one finds that the right hand side of \eqref{impart} equals
$$
-\Big(\frac{1}{2i}\tr B^* + {\rm Im}c^*\Big)\int_{\R^N}|v|^2 \, dz
$$
and therefore \eqref{cons} is proved as soon as we show that 
\begin{equation}\label{contaccio}
\frac{1}{2i}\tr B^* + {\rm Im}c^*=0\,.
\end{equation}
In order to do this we first rewrite the coefficients $c^*$ and $B^*$ in a 
suitable form. 
Denoting by $\langle\cdot\,,\cdot\rangle$ the Hermitian inner product in $L^2(\T^N)$ 
and using equation \eqref{deriv1t} we write

\begin{equation}\label{formc}
c^*=  \frac{1}{2i\pi}\langle 
\Ao_n\frac{\partial \psin}{\partial \theta_k},\frac{\partial \psin}{\partial x_k}\rangle
-\int_{\T^N}A_{1,k}(\nabla_y-2i\pi\theta^n)\bpsin\cdot\psin e_k\, dy\,,
\end{equation}
\no
while by equations \eqref{deriv1t}-\eqref{eq4} it follows that

\begin{align}\label{formB}
  \frac{1}{2i\pi}\frac{\partial^2 \lan}{\partial x_h\partial\theta_k}=
     & -\frac{1}{2i\pi}\langle\overline{
       \Ao_n\frac{\partial \psin}{\partial \theta_k},\frac{\partial \psin}{\partial x_h}}\rangle
       -\frac{1}{2i\pi}\langle\overline{
       \Ao_n\frac{\partial \psin}{\partial x_h},\frac{\partial \psin}{\partial\theta_k}}\rangle \\
\non & +2i{\rm Im}\int_{\T^N}A_{1,h}(\nabla_y-2i\pi\theta^n)\bpsin\cdot\psin e_k\, dy \,.
\end{align}
\no
By formulae \eqref{formc}-\eqref{formB} it is readily seen that equality
\eqref{contaccio} holds true.

In order to prove the self-adjointness of the operator $\Ao^*$, one first checks 
that $\Ao^*$ is symmetric, which easily follows by \eqref{contaccio} and 
the fact that ${\displaystyle \overline{B}^*=-B^*}$, and then observes that 
up to addition of a multiple of the identity the operator $\Ao^*$ is monotone 
(see {\em e.g.} \cite{brezis2}, Chapter VII).
\end{proof}

In the next proposition we will denote by $\nabla\nabla \lan$ the Hessian matrix of the function 
$\lan(x,\theta)$ evaluated at the point $(x^n,\theta^n)$, namely
$$
\nabla\nabla \lan =
\left(
\begin{array}{ll}
\nabla_{x}\nabla_{x}\lan      & \nabla_{\theta}\nabla_{x}\lan \\
\nabla_{\theta}\nabla_{x}\lan & \nabla_{\theta}\nabla_{\theta}\lan 
\end{array}
\right)
(x^n,\theta^n)\,.
$$

\begin{proposition}\label{localization}
Assume that the matrix $\nabla\nabla \lan$ is positive 
definite. Then there exists an orthonormal basis ${\displaystyle \{\var_n\}_{n\geq 1}}$ 
of eigenfunctions  
of $\Ao^*$; moreover for each $n$ there exists a real constant $\gamma_n>0$ such that
\begin{equation}\label{decay}
e^{\gamma_n |z|}\var_n\,,\,e^{\gamma_n |z|}\nabla\var_n  \in L^2(\R^N)\,.
\end{equation}
\end{proposition}

\begin{proof}
Up to shifting the spectrum of the operator $\Ao^{*}$, we may assume that Re$(c^*)=0$.
In order to prove the existence of an orthonormal basis of eigenfunctions 
we introduce the inverse operator of $\Ao^{*}$, denoted by $G^*$  
\begin{align}
\non G^{*}: L^2(\R^N) & \to L^2(\R^N) \\
\non               f  & \to \var \text{ unique solution in } H^1(\R^N) \text{ of} \\
\label{inverse}       &\quad\quad \Ao^* \var = f \quad\text{ in }\R^N
\end{align}
and we show that $G^*$ is compact. Indeed  multiplication of \eqref{inverse} by 
$\bar{\var}$ yields
\begin{equation}\label{compact}
\int_{\R^N}[
A^*\nabla\var\cdot\nabla\bar{\var}-iB^*{\rm Im}(\var z\cdot\nabla\bar{\var})+
D^*z\cdot z|\var|^2 ] \, dz=
\int_{\R^N}f\bar{\var}\, dz\,.
\end{equation}
Upon defining the $2N$-dimensional vector-valued function $\Phi$
$$
\Phi:=\left(\hspace{-2mm}
\begin{array}{c}
2i\pi z\var \\
\nabla\var
\end{array}\hspace{-2mm}
\right)
$$
we rewrite \eqref{compact} in agreement with this block notation
\begin{equation*}
\int_{\R^N}\frac{1}{8\pi^2} \nabla\nabla\lan \Phi\cdot\overline\Phi \, dz=
\int_{\R^N}f\bar{\var}\, dz \,. 
\end{equation*}
By the positivity assumption on the matrix $\nabla\nabla\lan$ it follows that there 
exists a positive constant $c_0$ such that
\begin{equation*}
c_0\Big(||\nabla\var||_{L^2(\R^N)}^2+||z\var||_{L^2(\R^N)}^2\Big)
\leq
||f||_{L^2(\R^N)}||\var||_{L^2(\R^N)}\,,
\end{equation*}
which implies by a standard argument
$$
||\var||_{L^2(\R^N)}^2 + ||\nabla\var||_{L^2(\R^N)}^2+||z\var||_{L^2(\R^N)}^2 \leq
C ||f||_{L^2(\R^N)}^2,
$$
from which we deduce the compactness of $G^*$ in $L^2(\R^N)$-strong.
Thus there exists an infinite countable number of eigenvalues for $\Ao^{*}$.

We are left to prove the exponential decay of the eigenfunctions (this
fact is quite standard, see {\em e.g.} \cite{alam}). Let $\var_n$ be an
eigenfunction and let $\sigma_n$ be the associated eigenvalue
\begin{equation}\label{eigenfn}
\Ao^* \var_n = \sigma_n \var_n\,.
\end{equation}  
Let $R_0>0$ and $\rho\in C^\infty(\R)$ be a real function such that
$0\leq\rho\leq 1$, $\rho(s)=0$ for $s\leq R_0$ and $\rho(s)=1$ for $s\geq R_0+1$ and 
for every positive integer $k$ define $\rho_k\in C^\infty(\R^N)$ in the following way
$$
\rho_k(z):=\rho(|z|-k).
$$
We now multiply \eqref{eigenfn} by $\bar{\var}_n\rho_k^2$ to get
\begin{equation*}
\int_{\R^N}\rho_k^2\left(
A^*\nabla\var_n\cdot\nabla\bar{\var}_n-iB^*{\rm Im}(\var_n z\cdot\nabla\bar{\var}_n)+
D^*z\cdot z|\var_n|^2 -\sigma_n|\var_n|^2 \right) dz=
\end{equation*}
\begin{equation}\label{trucco}
\int_{\R^N}\left(\rho_k|\var_n|^2B^*z\cdot\nabla\rho_k -
2\rho_k \,\bar{\var}_n A^*\nabla\var_n\cdot\nabla\rho_k \right) dz\,.
\end{equation}
Next remark that since the left hand side of \eqref{trucco} is real the right hand side 
must be also real and therefore it is equal to
\begin{equation}\label{real}
\int_{\R^N}- 2\rho_k \,{\rm Re}(\bar{\var}_n A^*\nabla\var_n)\cdot\nabla\rho_k  \, dz\,.
\end{equation} 
Let $B_k$ denote the ball of radius $R_0+k$ and center $z=0$ and observe 
that the support of $\nabla\rho_k$ is contained in $B_{k+1}\setminus B_k$.
Then putting up together \eqref{trucco} and \eqref{real} and using again the positive 
definiteness of the matrix $\nabla\nabla\lan$ we obtain for $R_0$ sufficiently large 
($\sqrt{R_0}>\sigma_n$ does the job)
\begin{equation*}
||\var_n||_{H^1(\R^N\setminus B_{k+1})}^2
\leq 
c_1\Big(||\var_n||^2_{H^1(\R^N\setminus B_k)}-||\var_n||^2_{H^1(\R^N\setminus B_{k+1})}\Big)
\end{equation*}
where $c_1$ is a positive constant independent of $k$.
Thus we deduce that 
\begin{equation}\label{estim}
||\var_n||_{H^1(\R^N\setminus B_{k+1})}^2
\leq 
\Big(\frac{c_1}{1+c_1}\Big)^k||\var_n||^2_{H^1(\R^N\setminus B_0)}\,.
\end{equation}
Upon defining a positive constant $\gamma_0>0$ by 
$$
\Big(\frac{c_1}{1+c_1}\Big)^k=e^{-2\gamma_0(k+R_0)}
$$
it is finally seen that \eqref{estim} implies the estimate \eqref{decay} for any 
exponent $0<\gamma_n<\gamma_0$.
\end{proof}

\no
{\bf Proof of Theorem \ref{mainth}.}
We rescale the space variable by introducing
$$
z=\frac{x}{\sqrt{\e}} \,,
$$
and define the sequence $v_{\e}$ by 
\begin{equation}\label{vdef}
v_{\e}(t,z):=e^{-i\frac{\lan t}{\e}}
e^{-2i\pi\frac{\theta^{n}\cdot x}{\e}}u_{\e}(t,x)\,.
\end{equation}
By the a priori estimates of Lemma \ref{apriori} 
it follows that $v_{\e}(t,z)$ satisfies
$$
|| v_\e ||_{L^\infty\left(\R^+;L^2(\R^N)\right)} + 
\sqrt{\e} || \nabla v_\e ||_ {L^\infty\left(\R^+;L^2(\R^N)\right)}
\leq C ,
$$
and applying the compactness of two-scale convergence 
(see Proposition \ref{prop2s}), up to a subsequence, there exists a limit 
$v^*(t,z,y)\in L^2\left(\R^+\times\R^N;H^1(\T^N)\right)$
such that $v_\e$ and $\sqrt{\e}\nabla v_\e$ two-scale 
converge to $v^*$ and $\nabla_y v^*$, respectively. 
Similarly, by definition of the initial data, 
$v_\e(0,z)$ two-scale converges to 
$\psi_n (y) v^0(z)$. 

\no
Although $v_\e$ is the unknown which will pass to the limit 
in the sequel, it is simpler to write an equation for another  
function, namely
\begin{equation}\label{wdef}
w_{\e}(t,z):=e^{2i\pi\frac{\theta^{n}\cdot z}{\sqrt{\e}}}v_{\e}(t,z)
= e^{-i\frac{\lan t}{\e}} u_{\e}(t,x)\,.
\end{equation}
By \eqref{wdef} it follows that 
\begin{equation}\label{formula}
\nabla w_{\e}=e^{2i\pi\frac{\theta^{n}\cdot z}{\sqrt{\e}}}
\Big(\nabla +2i\pi\frac{\theta^{n}}{\sqrt{\e}}\Big)v_{\e}\,,
\end{equation}
and it can be checked that the new unknown $w_{\e}$ solves the following equation 
\begin{equation}\label{neweq}
\left\{
\begin{array}{ll}
\displaystyle 
i \frac{\partial w_{\e}}{\partial t} - 
\div [A\left(\sqrt{\e}z,z/\sqrt{\e}\right)\nabla w_{\e}]
+ \frac{1}{\e} [ c(\sqrt{\e}z, z/\sqrt{\e}) - \lan ] w_{\e}=0  & \mbox{ in } \R^{N}\times \R^+\\
[3mm]
w_{\e}(0,z)=u^{0}_{\e}(\sqrt{\e}z) \hspace{5.1cm} & \mbox{ in } \R^{N}\\
\end{array}
\right.
\end{equation}
\no
where the differential operators div and $\nabla$ act with respect to 
the new variable $z$. 

\no 
{\bf First step.} 
We multiply the equation \eqref{neweq} by the complex conjugate of 
$$
\e\phi\Big(t,z,\frac{z}{\sqrt{\e}}\Big) e^{2i\pi\frac{\theta^n\cdot z}{\sqrt{\e}}} 
$$ 
where $\phi(s,z,y)$ is a smooth test function defined on 
$\R^+\times\R^N\times\T^N$, with compact support in $\R^+\times\R^N$.  
Since this test function has compact support (fixed with respect to $\e$), the effect of 
the non-periodic variable in the coefficients is negligible for sufficiently small $\e$. 
Therefore we can replace the value 
of each coefficient at  $(\sqrt{\e}z, z/\sqrt{\e})$ by its Taylor expansion of  order two 
about the point $(0, z/\sqrt{\e})$.
Integrating by parts and using \eqref{wdef} and \eqref{formula} yields
\begin{align*}
& -i \e \int_{0}^{+\infty}\hspace{-3mm} \int_{\R^{N}} 
   v_{\e} \frac{\partial \bphi^{\e}}{\partial t}\: dt\, dz 
   -  i \e \int_{\R^{N}} 
   v_{\e}(0,z)\bphi \Big(0,z,\frac{z}{\sqrt{\e}}\Big) \, dz  \\
& + \int_{0}^{+\infty}\hspace{-3mm}
   \int_{\R^{N}}
   [ \A + \A_{1,h} \, \sqrt{\e} z_{h} + \textstyle{\frac{1}{2}} \A_{2,lh} \, \e z_{l}  z_{h} 
   + o(\e)]
   (\sqrt{\e}\nabla +2i\pi\theta^{n})v_{\e}\hspace{-1mm}  \cdot
   (\sqrt{\e}\nabla-2i\pi\theta^{n})\bphi^{\e}  \,dz \, dt \\
& + \int_{0}^{+\infty}\hspace{-3mm} \int_{\R^{N}}
  [\ce + \ce_{1,h} \sqrt{\e} z_{h} + \textstyle{\frac{1}{2}} \ce_{2,lh}\,\e z_{l}  z_{h} 
  +o(\e)- \lan ] \, v_{\e} \bphi^{\e}
 \:dz \, dt   = 0 .
\end{align*} 
Passing to the two-scale limit we get the variational formulation 
of 
$$
- (\div_y +2i\pi\theta^n) \Big( A(y) (\nabla_y +2i\pi\theta^n)v^* \Big) + 
c(y)v^* = \lan v^* \quad \mbox{ in } \T^N .
$$
The simplicity of $\lan$ implies that there exists a scalar function $v(t,z)\in 
L^2\left(\R^+\times\R^N\right)$ such that 
\begin{equation}\label{vstar}
v^*(t,z,y) = v(t,z) \psi_n(y) . 
\end{equation}

\no
{\bf Second step.} 
We multiply (\ref{neweq}) by the complex conjugate of 
\begin{equation*}
\Psi_{\e}(t,z)=e^{2i\pi\theta^{n}\cdot\frac{z}{\sqrt{\e}}}\Big[
\psie \phi(t,z)  +
\sqrt{\e}
\sum_{k=1}^{N}
\Big(
\frac{1}{2i\pi}
\frac{\partial \psie}{\partial \theta_{k}}
\frac{\partial \phi}{\partial z_{k}}(t,z)  +
z_{k}
\frac{\partial \psie}{\partial x_{k}}
\phi(t,z)
\Big)
\Big]\,,
\end{equation*}
where $\phi(t,z)$ is a smooth test function with compact support 
in $\R^+\times\R^N$.
\no
We first look at those terms of the equation involving time derivatives:
\begin{align}\label{timeder}
     & \int_{0}^{+\infty}\hspace{-3mm}\int_{\R^{N}}
       i\frac{\partial w_\e}{\partial t}\bar{\Psi}_\e \, dt\,dz =\\
\non &\int_{0}^{+\infty}\hspace{-3mm}\int_{\R^{N}}-i v_{\e} 
       \left[ 
       \bpsie \frac{\partial \bphi}{\partial t} +\sqrt{\e}\sum_{k=1}^{N}
       \left(
      -\frac{1}{2i\pi}\frac{\partial \bpsie}{\partial \theta_{k}}
       \frac{\partial^{2} \bphi}{\partial t \partial z_{k}} + z_{k}
       \frac{\partial \bpsie}{\partial x_{k}}\frac{\partial \bphi}{\partial t}
       \right)
       \right]\, 
       dt\, dz  \\
\non  & -i \int_{\R^{N}}
   v_{\e}(0,z)
   \left[
   \bpsie \bphi (0,z) +
   \sqrt{\e} \sum_{k=1}^{N}
   \left(
   -\frac{1}{2i\pi}
   \frac{\partial \bpsie}{\partial \theta_{k}}
   \frac{\partial \bphi}{\partial z_{k}}(0,z) + 
   z_{k} \frac{\partial \bpsie}{\partial x_{k}}\bphi (0,z)
   \right)
   \right] \, dz\,.  
\end{align}
\no
Recalling the normalization $\int_{\T^{N}}|\psin|^{2}\,dy=1$,  
we find that the two-scale limit of the term on the left hand side of 
\eqref{timeder} is given by the expression
\begin{equation}\label{term1}
-i  \int_{0}^{+\infty}\hspace{-3mm}\int_{\R^{N}} v \frac{\partial \bphi}{\partial t}\: dz\, dt 
-i \int_{\R^{N}} v^{0}\bphi (0,z) \, dz  \,.
\end{equation}

\no
We further decompose $\Psi_\e$ as follows 
$$
\Psi_\e=\Psi_\e^1+\Psi_\e^2\cdot z \quad\text{ with }\quad 
\Psi_\e^2 = \sqrt{\e}
e^{2i\pi\theta^{n}\cdot\frac{z}{\sqrt{\e}}}
\sum_{k=1}^{N}\frac{\partial \psie}{\partial x_{k}}\phi(t,z)e_k.
$$
\no
Getting rid of all terms multiplied by $o(\e)$ and taking into account \eqref{wdef} 
and \eqref{formula} we next pass to the limit in the remaining terms 
of \eqref{neweq} multiplied by $\bar{\Psi}_{\e}$. 
The computation is similar to \cite{alpiat2} but it involves new terms since 
$\psin$ and its derivatives also depend on $x$. 
We first look at those terms which are of zero order with respect to $z$, namely
\begin{align}\label{zero}
     & \int_{0}^{+\infty}\hspace{-3mm}\int_{\R^{N}}\Big[
      \A \nabla w_\e \cdot (\nabla\bar{\Psi}_\e^1+ \bar{\Psi}_\e^2) +
      \frac{1}{\e}(\ce-\lan)w_\e \bar{\Psi}_\e^1 \Big] \: dz\, dt \\
\non  & = \int_{0}^{+\infty}\hspace{-3mm}\int_{\R^{N}}\Big[
         \frac{1}{\e} \A\Big(\sqrt{\e}\nabla + 2i\pi\theta^{n}\Big)v_{\e}
         \cdot (\nabla_{y}-2i\pi\theta^{n})\bpsie \bphi +
          \frac{1}{\e}(\ce-\lan)\bpsie v_{\e} \bphi  
          \Big] \: dz\, dt\\
\non  & -\frac{1}{2i\pi}
          \int_{0}^{+\infty}\hspace{-3mm}\int_{\R^{N}}\Big[
           \frac{1}{\sqrt{\e}}\A \Big(\sqrt{\e}\nabla + 2i\pi\theta^{n}\Big)v_{\e} \cdot  
           (\nabla_{y}-2i\pi\theta^{n})\frac{\partial \bpsie}{\partial \theta_{k}}
           \frac{\partial \bphi}{\partial z_{k}} \\
\non  & \hspace{2.7cm} + \frac{1}{\sqrt{\e}} (\ce-\lan) v_{\e}
            \frac{\partial \bpsie}{\partial \theta_{k}}\frac{\partial \bphi}{\partial z_{k}}
           \Big] \:dz\, dt \\
\non  &   +\int_{0}^{+\infty}\hspace{-3mm}\int_{\R^{N}} 
           \frac{1}{\sqrt{\e}}\A \Big(\sqrt{\e}\nabla + 2i\pi\theta^{n}\Big)v_{\e} \cdot 
           \bpsie \nabla \bphi  \: dz\, dt \\
\non   &  + \int_{0}^{+\infty}\hspace{-3mm}\int_{\R^{N}} -\frac{1}{2\pi i}
            \A \Big(\sqrt{\e}\nabla + 2i\pi\theta^{n}\Big)v_{\e} \cdot 
            \frac{\partial \bpsie}{\partial \theta_{k}}
            \nabla\frac{\partial \bphi}{\partial z_{k}} \:dz\, dt\,\\
\non   &  + \int_{0}^{+\infty}\hspace{-3mm}\int_{\R^{N}} 
           \A \Big(\sqrt{\e}\nabla + 2i\pi\theta^{n}\Big)v_{\e} 
            \cdot \frac{\partial \bpsie}{\partial x_{k}}\bphi \, e_{k}
            \:dz\, dt\,.
\end{align}

\vspace{2mm}

\no
Using equation \eqref{eq1} with $\var=v_{\e}\bphi$ and equation \eqref{eq2} with 
${\displaystyle \var=v_{\e}\frac{\partial \bphi}{\partial z_{k}}}$ we rewrite the first 
two integrals in the right hand side of \eqref{zero} as follows 

\begin{align*}
&   \int_{0}^{+\infty}\hspace{-3mm}\int_{\R^{N}}
    -\frac{1}{\sqrt{\e}} \A(\nabla_{y}-2i\pi\theta^{n})\bpsie\cdot v_{\e}\nabla \bphi 
    \,dz\, dt\,  \\
&   +\int_{0}^{+\infty}\hspace{-3mm}\int_{\R^{N}} \Big[
     \frac{1}{2i\pi} \A(\nabla_{y}-2i\pi\theta^{n})\frac{\partial \bpsie}{\partial \theta_{k}}
     \cdot  v_{\e}\nabla \frac{\partial \bphi}{\partial z_{k}}
     + \frac{1}{\sqrt{\e}}\A e_{k}\cdot 
     v_{\e} \frac{\partial \bphi}{\partial z_{k}}(\nabla_{y}-2i\pi\theta^{n}) \bpsie \\ 
&   \hspace{2cm} - \frac{1}{\sqrt{\e}}\A \bpsie e_{k}\cdot 
     \Big(\sqrt{\e}\nabla +2i\pi\theta^{n}\Big)
     \Big(v_{\e}\frac{\partial \bphi}{\partial z_{k}}\Big)\Big]  \,dz\, dt\,.            
\end{align*}

\vspace{3mm}

\no
Combining the above terms with the other terms in \eqref{zero} and passing to the 
two-scale limit in \eqref{zero} yields

\begin{align}\label{term2}
     & \int_{0}^{+\infty}\hspace{-3mm}\int_{\R^{N}}\int_{\T^{N}}
          \Big[
          \frac{1}{2i\pi}A\psin (\nabla_{y}-2i\pi\theta^{n})\frac{\partial\bpsin}{\partial\theta_{k}}
         -\frac{1}{2i\pi}A \frac{\partial\bpsin}{\partial\theta_{k}}(\nabla_{y}+2i\pi\theta^{n})\psin
         - A |\psin|^{2} e_{k}\Big] \\
\non &  \hspace{2.4cm}\cdot v \nabla \frac{\partial \bphi}{\partial z_{k}}
          \:dy \,dz\, dt \\
\non &  +\int_{0}^{+\infty}\hspace{-3mm}\int_{\R^{N}}\int_{\T^{N}}
         A(\nabla_{y}+2i\pi\theta^{n})\psin\cdot \frac{\partial \bpsin}{\partial x_{k}}v\bphi \, e_{k}
          \:dy \,dz\, dt\,.
\end{align}
By equation \eqref{eq6} it can be seen that the first integral of \eqref{term2} equals
\begin{equation}\label{tensorA}
\int_{0}^{+\infty}\hspace{-3mm}\int_{\R^{N}} A^{*}\nabla v \nabla \bphi \,dz\, dt\,.
\end{equation}

\no
We now focus on those terms which are linear in $z$:
\begin{align}\label{linear}
\non & \int_{0}^{+\infty}\hspace{-3mm}\int_{\R^{N}}\Big[
       \A\nabla w_\e\cdot(\nabla\bar{\Psi}_\e^2 z)+\frac{1}{\e}(\ce-\lan)w_\e\bar{\Psi}_\e^2 z  
       +\A_{1,k}\sqrt{\e}z_k\nabla w_\e\cdot (\nabla\bar{\Psi}_\e ^1+\bar{\Psi}_\e^2)\\
\non & \hspace{1.5cm} +\frac{1}{\sqrt{\e}}\ce_{1,k}z_k w_{\e}\bar{\Psi}_e^1
        \Big] \: dz \, dt  \\
\non & = \int_{0}^{+\infty}\hspace{-3mm}\int_{\R^{N}}\Big[
         \frac{1}{\sqrt{\e}} \A \Big(\sqrt{\e}\nabla +2i\pi\theta^{n}\Big)v_{\e} \cdot 
          (\nabla_{y}-2i\pi\theta^{n})\frac{\partial \bpsie}{\partial x_{k}}\,\bphi z_{k}
          + \frac{1}{\sqrt{\e}}(\ce-\lan)v_{\e}\frac{\partial \bpsie}{\partial x_{k}}\,\bphi z_{k}
          \Big] dz \, dt  \\
\non   & + \int_{0}^{+\infty}\hspace{-3mm}\int_{\R^{N}}\Big[
           \frac{1}{\sqrt{\e}}\A_{1,k} \Big(\sqrt{\e}\nabla +2i\pi\theta^{n}\Big)v_{\e}\hspace{-1mm}
          \cdot \hspace{-1mm}
         (\nabla_{y}-2i\pi\theta^{n}) \bpsie \,\bphi z_{k}
          + \frac{1}{\sqrt{\e}} \ce_{1,k} v_{\e}\bpsie \, \bphi \, z_{k}   
          \Big]\: dz \, dt  \\
\non   & + 
         \int_{0}^{+\infty}\hspace{-3mm}\int_{\R^{N}}\Big[ 
          \A(\sqrt{\e}\nabla +2i\pi\theta^{n})v_{\e}\cdot 
          \frac{\partial \bpsie}{\partial x_{k}}\nabla\bphi\, z_{k} 
           + \A_{1,k} (\sqrt{\e}\nabla +2i\pi\theta^{n})v_{\e}
           \cdot  \bpsie \nabla\bphi \, z_{k}\Big]  \: dz \, dt \\
\non   & -\frac{1}{2i\pi}\int_{0}^{+\infty}\hspace{-3mm}\int_{\R^{N}}\Big[ 
           \A_{1,h} (\sqrt{\e}\nabla +2i\pi\theta^{n})v_{\e}\cdot 
           (\nabla_{y}-2i\pi\theta^{n})
           \frac{\partial \bpsin}{\partial \theta_{k}}\frac{\partial \bphi}{\partial z_{k}}\, z_{h} 
           + \ce_{1,h} v_{\e}
           \frac{\partial \bpsin}{\partial \theta_{k}}\frac{\partial \bphi}{\partial z_{k}}\, z_{h} 
           \Big]\: dz \, dt \\          
\non &  +\int_{0}^{+\infty}\hspace{-3mm}\int_{\R^{N}}\Big[ 
            \sqrt{\e}\A_{1,h}(\sqrt{\e}\nabla +2i\pi\theta^{n})v_{\e}\cdot 
            \Big(-\frac{1}{2i\pi}
            \frac{\partial \bpsie}{\partial \theta_{k}}\nabla\frac{\partial \bphi}{\partial z_{k}}
            +\frac{\partial \bpsie}{\partial x_{k}}\bphi \, e_{k}
            \Big) z_{h}  
            \Big]\: dz \, dt \,. \\     
\end{align}
By equation \eqref{eq3} with ${\displaystyle \var = v_{\e}\bphi z_{k}}$  
it can be seen that the sum of the first two integrals in the right hand side of \eqref{linear} gives
\begin{equation}\label{byeq3}
- \int_{0}^{+\infty}\int_{\R^{N}} 
\A(\nabla_{y}-2i\pi\theta^{n})\frac{\partial \bpsie}{\partial x_{k}} \cdot
v_{\e} \nabla (\bphi z_{k}) 
+ \A_{1,k} (\nabla_{y}-2i\pi\theta^{n}) \bpsie \cdot  v_{\e} \nabla (\bphi z_{k})
\Big) \: dz \, dt \,.
\end{equation}          

\no
Therefore passing to the two-scale limit in \eqref{linear} we find

\begin{align} \label{linear1}
       & -\int_{0}^{+\infty}\hspace{-3mm}\int_{\R^{N}}\int_{\T^{N}}\Big[
           A (\nabla_{y}-2i\pi\theta^{n})\frac{\partial \bpsin}{\partial x_{k}}
           \cdot v \psin \bphi\, e_{k}
           + A_{1,k} (\nabla_{y}-2i\pi\theta^{n}) \bpsin\cdot 
           v \psin \bphi \, e_{k}\Big]  \: dy \, dz \, dt\\
\non    & - \int_{0}^{+\infty}\hspace{-3mm}\int_{\R^{N}}\int_{\T^{N}}\Big[
           A (\nabla_{y}-2i\pi\theta^{n})\frac{\partial \bpsin}{\partial x_{k}}
           \cdot v \psin z_{k} \nabla\bphi            
           + A_{1,k}(\nabla_{y}-2i\pi\theta^{n})\bpsin \cdot 
           v \psin z_{k} \nabla\bphi\Big] \: dy \, dz \, dt\\
\non   & + \int_{0}^{+\infty}\hspace{-3mm}\int_{\R^{N}}\int_{\T^{N}}\Big[
           A(\nabla_{y}+2i\pi\theta^{n})\psin \cdot v
           \frac{\partial \bpsin}{\partial x_{k}} z_{k}\nabla\bphi 
           + A_{1,k}(\nabla_{y}+2i\pi\theta^{n})\psin \cdot 
           v \bpsin  z_{k}\nabla\bphi \Big]\: dy \, dz \, dt\\                       
\non   & -\frac{1}{2i\pi}
          \int_{0}^{+\infty}\hspace{-3mm}\int_{\R^{N}}\int_{\T^{N}}\Big[
           A_{1,h}(\nabla_{y}+2i\pi\theta^{n})\psin \cdot 
           (\nabla_{y}-2i\pi\theta^{n})\frac{\partial \bpsin}{\partial \theta_{k}}
           v z_{h}\frac{\partial \bphi}{\partial z_{k}}\\
\non  &  \hspace{3cm} + c_{1,h} \psin 
           \frac{\partial \bpsin}{\partial \theta_{k}}v z_{h}\frac{\partial \bphi}{\partial z_{k}}
           \Big]\: dy \, dz \, dt \,.
\end{align}          
\no
By equation \eqref{eq4} it follows that the last integral in \eqref{linear1} is equal to
\begin{align}\label{linear2}
       &   \int_{0}^{+\infty}\hspace{-3mm}\int_{\R^{N}}\int_{\T^{N}}
            \Big[
            A_{1,h}\psin e_{k} \cdot(\nabla_{y}-2i\pi\theta^{n})\bpsin + A\psin e_{k} 
             \cdot(\nabla_{y}-2i\pi\theta^{n})\frac{\partial \bpsin}{\partial x_{h}}\psin \Big]
             v z_{h} \frac{\partial \bphi}{\partial z_{k}}
             \:\, dy \, dz \, dt \\                               
\non   & - \int_{0}^{+\infty}\hspace{-3mm}\int_{\R^{N}}\int_{\T^{N}}
           \Big[
           A_{1,h}\bpsin e_{k}\cdot (\nabla_{y}+2i\pi\theta^{n})\psin  
           + A\frac{\partial \bpsin}{\partial x_{h}}e_{k}\cdot (\nabla_{y}+2i\pi\theta^{n})\psin\Big] 
           v z_{h} \frac{\partial \bphi}{\partial z_{k}}
           \:\, dy \, dz \, dt \\
\non   & - \int_{0}^{+\infty}\hspace{-3mm}\int_{\R^{N}}\int_{\T^{N}}
           \frac{1}{2i\pi}\frac{\partial^{2} \lan}{\partial x_{h}\partial\theta_{k}}
           |\psin|^{2} v z_{h} \frac{\partial \bphi}{\partial z_{k}}
           \:\, dy \, dz \, dt \,.
\end{align}

\no
Next notice that the first and the second line of \eqref{linear2} cancel out with the second 
and the third line of \eqref{linear1} respectively and therefore \eqref{linear1} reduces to
\begin{align}\label{term3}
       &  -\int_{0}^{+\infty}\hspace{-3mm}\int_{\R^{N}}\int_{\T^{N}}\Big[  
           A (\nabla_{y}-2i\pi\theta^{n})\frac{\partial \bpsin}{\partial x_{k}}
           \cdot v\psin\bphi \,e_{k}
           + A_{1,k}(\nabla_{y}-2i\pi\theta^{n})\bpsin\cdot v \psin 
           \bphi \, e_{k}\Big]  \:\, dy \, dz \, dt \\
\non &   -\int_{0}^{+\infty}\hspace{-3mm}\int_{\R^{N}}
            \frac{1}{2i\pi}\frac{\partial^{2} \lan}{\partial x_{h}\partial\theta_{k}}
             v \frac{\partial \bphi}{\partial z_{k}} \, z_{h}
            \: dz \, dt \,.
\end{align}

\no
Finally we consider all quadratic in $z$ terms:
\begin{align}
\non    &   \frac{1}{2}\int_{0}^{+\infty}\hspace{-3mm}\int_{\R^{N}}\Big[
             \A_{2,lh} \,\e z_{l}z_{h}\nabla w_\e\cdot(\nabla\bar{\Psi}_\e^1+\bar{\Psi}_\e^2)
            + \ce_{2,lh} z_{l}  z_{h}  w_{\e}\bar{\Psi_e^1}\Big] \: dz \, dt  \\ 
\non    &   +\int_{0}^{+\infty}\hspace{-3mm}\int_{\R^{N}}\Big[
            \A_{1,k}\sqrt{\e}z_k\nabla w_\e\cdot (z\nabla\bar{\Psi}_\e ^2)
            +\frac{1}{\sqrt{\e}}\ce_{1,k}z_k w_{\e}z\cdot\bar{\Psi}_\e^2 \Big]\: dz \, dt  \\
\non    &   = \frac{1}{2} \int_{0}^{+\infty}\hspace{-3mm}\int_{\R^{N}}
              \A_{2,lh} \, \sqrt{\e} z_{l}z_{h}
              \Big(\sqrt{\e}\nabla+2i\pi\theta^{n}\Big)v_{\e}\cdot
              \Big[
              \frac{1}{\sqrt{\e}}(\nabla_{y}-2i\pi\theta^{n})\bpsie \bphi  
              +\bpsie \nabla\bphi \Big] \: dz \, dt \\
\non    &  - \frac{1}{2}\int_{0}^{+\infty}\hspace{-3mm}\int_{\R^{N}}  
              \A_{2,lh} \, \sqrt{\e} z_{l}z_{h}
              \Big(\sqrt{\e}\nabla+2i\pi\theta^{n}\Big)v_{\e} \\
\non    &  \hspace{2.4cm}\cdot
              \Big[
              \frac{1}{2\pi i}  
              \nabla_{y}\frac{\partial \bpsie}{\partial \theta_{k}} 
              \frac{\partial \bphi}{\partial z_{k}}
             +\sqrt{\e}
              \Big(
               \frac{1}{2i\pi}
              \frac{\partial \bpsie}{\partial \theta_{k}} \nabla\frac{\partial \bphi}{\partial z_{k}}
              + e_{k} \frac{\partial \bpsie}{\partial x_{k}} \bphi
              \Big) \Big] \, dz \, dt \\
\non    &  + \int_{0}^{+\infty}\hspace{-3mm}\int_{\R^{N}}
              \A_{1,h}\,z_{h}\Big(\sqrt{\e}\nabla+2i\pi\theta^{n}\Big)v_{\e}
              \cdot
              \Big[
              z_{k}(\nabla_{y}-2i\pi\theta^{n})\frac{\partial \bpsie}{\partial x_{k}} \bphi
              +\sqrt{\e} z_{k} \frac{\partial \bpsie}{\partial x_{k}} \nabla \bphi
              \Big]  \: \, dz \, dt \\
\non    &  + \int_{0}^{+\infty}\hspace{-3mm}\int_{\R^{N}}
              \frac{1}{2} \ce_{2,lh} z_{l}  z_{h}  v_{\e}
              \Big(
              \bpsie\bphi  
              - \sqrt{\e} \frac{1}{2i\pi} 
              \frac{\partial \bpsie}{\partial \theta_{k}} \frac{\partial \bphi}{\partial z_{k}} 
              \Big)  \: \, dz \, dt \\
\non    &  + \int_{0}^{+\infty}\hspace{-3mm}\int_{\R^{N}}
             \ce_{1,h} z_{h} v_{\e}
              z_{k} \frac{\partial \bpsie}{\partial x_{k}} \bphi
              \: dz \, dt   
\end{align}

\no
which give on passing to the two-scale limit
\begin{align} \label{quad}
          &  \frac{1}{2}\int_{0}^{+\infty}\hspace{-3mm}\int_{\R^{N}} \int_{\T^{N}}
              \Big[
              A_{2,lh}(\nabla_{y}+2i\pi\theta^{n})\psin\cdot (\nabla_{y}-2i\pi\theta^{n})\bpsin  
              + c_{2,lh} \psin \bpsin \Big] v \bphi \,  z_{l}  z_{h} \: dy \, dz \, dt  \\
\non    &  + \int_{0}^{+\infty}\hspace{-3mm}\int_{\R^{N}} \int_{\T^{N}}
              \Big[
              A_{1,h}(\nabla_{y}+2i\pi\theta^{n})\psin  \cdot   
              (\nabla_{y}-2i\pi\theta^{n})\frac{\partial \bpsin}{\partial x_{k}}             
              + c_{1,h} \psin \frac{\partial \bpsin}{\partial x_{k}} 
              \, \Big] v \bphi \,  z_{h}  z_{k}  \: dy \, dz \, dt  \,.
\end{align}
\no
Now using equation \eqref{eq5} we find that \eqref{quad} reduces itself to 

\begin{equation}\label{term4}
\int_{0}^{+\infty}\hspace{-3mm}\int_{\R^{N}} 
\frac{1}{2}\frac{\partial^{2} \lan}{\partial x_{l}\partial x_{h}} \:
v \bphi \,  z_{l}  z_{h} \: dz \, dt \,.
\end{equation}
\no
Summing up together \eqref{term1}, \eqref{term2}, \eqref{tensorA}, 
\eqref{term3} and \eqref{term4} 
yields the weak formulation of \eqref{hom}. 
By uniqueness of the solution of the homogenized problem \eqref{hom}, 
we deduce that the entire sequence $v_\e$ two-scale converges 
weakly to $\psin(y)v(t,x)$. 

It remains to prove the strong two-scale convergence of $v_\e$. 
By Lemma \ref{apriori} we have 
$$
|| v_\e(t) ||_{L^2(\R^N)} = || u_\e(t) ||_{L^2(\R^N)} = 
|| u^0_\e ||_{L^2(\R^N)} \to || \psi_n v^0 ||_{L^2(\R^N\times\T^N)}
= || v^0 ||_{L^2(\R^N)}
$$
by the normalization condition of $\psi_n$. From the conservation of 
energy of the homogenized equation \eqref{hom} we have 
$$
|| v(t) ||_{L^2(\R^N)} = || v^0 ||_{L^2(\R^N)} ,
$$
and thus we deduce the strong convergence from Proposition \ref{prop2s}. 
$\Box$

\bigskip

\begin{remark}\label{lastrem}
{\rm As usual in periodic homogenization \cite{allaire}, \cite{blp}, 
the choice of the test function $\Psi_\e$, in the proof 
of Theorem \ref{mainth}, is dictated by the formal two-scale 
asymptotic expansion that can be obtained for the 
solution $w_\e$ of \eqref{neweq}, namely
$$
w_\e(t,z) \approx 
e^{2i\pi\theta^{n}\cdot\frac{z}{\sqrt{\e}}}\Big[
\psin\Big(\frac{z}{\sqrt{\e}}\Big) v(t,z)  +
\sqrt{\e}
\sum_{k=1}^{N}
\Big(
\frac{1}{2i\pi}
\frac{\partial \psin}{\partial \theta_{k}}\Big(\frac{z}{\sqrt{\e}}\Big)   
\frac{\partial v}{\partial z_{k}}(t,z)  +
z_{k}
\frac{\partial \psin}{\partial x_{k}}\Big(\frac{z}{\sqrt{\e}}\Big)v(t,z)
\Big)
\Big]
$$
where $v$ is the homogenized solution of \eqref{hom}. 
Actually the homogenized equation that one gets by the asymptotic
expansion method is
\begin{equation}\label{asym}
i \frac{\partial v}{\partial t} - \div\left(A^{*}\nabla v \right) 
+ B^{*} \nabla v \cdot z + \bar{c}^{*} v 
+ v D^{*} z \cdot z  = 0 \, ,
\end{equation}
which apparently differs from \eqref{hom} by the following
zero-order term
$$
\left( \tr(\nabla_{\theta}\nabla_{x}\lan) - 4\pi \text{Im}( c^{*}) \right) v \,.
$$
By virtue of \eqref{contaccio} the above term vanishes,
so that formulae \eqref{asym} and \eqref{hom} are equivalent.
}
\end{remark}

\medskip
\centerline{\sc Acknowledgments}
This work was done while M. Palombaro was post-doc at the Centre de
Math\'ematiques Appliqu\'ees of Ecole Polytechnique. The hospitality
of people there is gratefully acknowledged. This work was partly
supported by the MULTIMAT european network MRTN-CT-2004-505226 funded by the EEC.

\end{document}